\documentclass[11pt,a4paper]{article}

\usepackage{inputenc}
\usepackage{amsmath}
\usepackage{bm}
\usepackage{bbold}
\usepackage{amsthm}
\usepackage{enumerate}

\usepackage{hyperref}

\title{Complete Solution of a Constrained Tropical Optimization Problem with Application to Location Analysis\thanks{Relational and Algebraic Methods in Computer Science, P. Hoefner, P. Jipsen, W. Kahl, M. E. Mueller, eds., vol. 8428 of Lecture Notes in Computer Science, pp. 362-378, Springer, 2014.}}

\author{Nikolai Krivulin\thanks{Faculty of Mathematics and Mechanics, Saint Petersburg State University, 28 Universitetsky Ave., Saint Petersburg, 198504, Russia, 
nkk@math.spbu.ru.}\thanks{The work was supported in part by the Russian Foundation for Humanities under Grant \#13-02-00338.}
}

\date{}

\newtheorem{theorem}{Theorem}
\newtheorem{lemma}[theorem]{Lemma}
\newtheorem{corollary}[theorem]{Corollary}

\begin{document}

\maketitle

\begin{abstract}
We present a multidimensional optimization problem that is formulated and solved in the tropical mathematics setting. The problem consists of minimizing a nonlinear objective function defined on vectors over an idempotent semifield by means of a conjugate transposition operator, subject to constraints in the form of linear vector inequalities. A complete direct solution to the problem under fairly general assumptions is given in a compact vector form suitable for both further analysis and practical implementation. We apply the result to solve a multidimensional minimax single facility location problem with Chebyshev distance and with inequality constraints imposed on the feasible location area. 
\\

\\

\end{abstract}

\section{Introduction}

Tropical (idempotent) mathematics encompasses various aspects of the theory and applications of semirings with idempotent addition and has its origin in a few pioneering works by Pandit \cite{Pandit1961Anew}, Cuninghame-Green \cite{Cuninghamegreen1962Describing}, Giffler \cite{Giffler1963Scheduling}, Vorob{'}ev \cite{Vorobjev1963Theextremal} and Romanovski{\u\i} \cite{Romanovskii1964Asymptotic}. At the present time, the literature on the topic contains several monographs, including those by Carr{\'e} \cite{Carre1979Graphs}, Cuninghame-Green \cite{Cuninghamegreen1979Minimax}, U.~Zimmermann \cite{Zimmermann1981Linear}, Baccelli et al. \cite{Baccelli1993Synchronization}, Kolokoltsov and Maslov \cite{Kolokoltsov1997Idempotent}, Golan \cite{Golan2003Semirings}, Heidergott, Olsder and van der Woude \cite{Heidergott2006Maxplus}, Gondran and Minoux \cite{Gondran2008Graphs}, and Butkovi\v{c} \cite{Butkovic2010Maxlinear}; as well as a rich variety of contributed papers.

Optimization problems that are formulated and solved in the tropical mathematics setting come from various application fields and form a noteworthy research domain within the research area. Certain optimization problems have appeared in the early paper \cite{Cuninghamegreen1962Describing}, and then the problems were investigated in many works, including \cite{Cuninghamegreen1979Minimax,Zimmermann1981Linear,Gondran2008Graphs,Butkovic2010Maxlinear}.

Tropical mathematics provides a useful framework for solving optimization problems in location analysis. Specifically, a solution in terms of tropical mathematics has been proposed by Cuninghame-Green \cite{Cuninghamegreen1991Minimax,Cuninghamegreen1994Minimax} to solve single facility location problems defined on graphs. A different but related approach to location problems on graphs and networks has been developed by K.~Zimmermann \cite{Zimmermann1992Optimization}, Hudec and K.~Zimmermann \cite{Hudec1993Aservice,Hudec1999Biobjective}, Tharwat and K.~Zimmermann \cite{Tharwat2010Oneclass} on the basis of the concept of max-separable functions.
 
Multidimensional minimax location problems with Chebyshev distance arise in various applications, including the location of emergency service facility in urban planning and the location of a component on a chip in electronic circuit manufacturing (see, e.g., Hansen, Peeters and Thisse \cite{Hansen1980Location,Hansen1981Constrained}). The two-dimensional problems on the plane without constraints can be solved directly on the basis of geometric arguments, as demonstrated by Sule \cite{Sule2001Logistics} and Moradi and Bidkhori \cite{Moradi2009Single}. The solution of the multidimensional constrained problems is less trivial and requires different approaches. These problems can be solved, for instance, by using standard linear programming techniques which, however, generally offer iterative procedures and do not guarantee direct solutions.

A strict tropical mathematics approach to solve both unconstrained and constrained minimax location problems with Chebyshev distance was developed by Krivulin \cite{Krivulin2011Analgebraic,Krivulin2012Anew}, and Krivulin and K.~Zimmermann \cite{Krivulin2013Direct}. The main result of \cite{Krivulin2011Analgebraic} is a direct solution to the unconstrained problem obtained by using the spectral properties of matrices in idempotent algebra. The application of another technique in \cite{Krivulin2012Anew,Krivulin2013Direct}, which is based on the derivation of sharp bounds on the objective function, shows that the solution in \cite{Krivulin2011Analgebraic} is complete.

In this paper, a new minimax Chebyshev location problem with an extended set of constraints is taken to both motivate and illustrate the development of the solution to a new general tropical optimization problem. The problem is to minimize a nonlinear objective function defined on vectors over a general idempotent semifield by means of a conjugate transposition operator. The problem involves constraints imposed on the solution set in the form of linear vector inequalities given by a matrix, and two-sided boundary constraints.

To solve the problem, we use the approach, which is proposed in \cite{Krivulin2013Amultidimensional,Krivulin2014Aconstrained} and combines the derivation of a sharp bound on the objective function with the solution of linear inequalities. The approach is based on the introduction of an auxiliary variable as a parameter, and the reduction of the optimization problem to the solution of a parametrized system of linear inequalities. Under fairly general assumptions, we obtain a complete direct solution to the problem and represent the solution in a compact vector form. The obtained result is then applied to solve the Chebyshev location problem, which motivated this study.

The paper is organized as follows. In Section~\ref{S-PDN}, we offer an introduction to idempotent algebra to provide a formal framework for the study in the rest of the paper. Section~\ref{S-SLN} offers the preliminary results on the solution of linear inequalities, which form a basis for later proofs. The main result is included in Section~\ref{S-OP}, which starts with a discussion of previously solved problems. Furthermore, we describe the problem under study, present a complete direct solution to the problem, consider particular cases, and give illustrative examples. Finally, application of the results to location analysis is discussed in Section~\ref{S-ALA}.

\section{Preliminary Definitions and Notation}
\label{S-PDN}

We start with a short, concise introduction to the key definitions, notation, and preliminary results in idempotent algebra, which is to provide a proper context for solving tropical optimization problems in the subsequent sections. The introduction is mainly based on the notation and results suggested in \cite{Krivulin2006Solution,Krivulin2009Onsolution,Krivulin2012Anew,Krivulin2013Amultidimensional}, which offer strong possibilities for deriving direct solutions in a compact form. Further details on both introductory and advanced levels are available in various works published on the topic, including \cite{Cuninghamegreen1979Minimax,Carre1979Graphs,Zimmermann1981Linear,Baccelli1993Synchronization,Kolokoltsov1997Idempotent,Golan2003Semirings,Heidergott2006Maxplus,Akian2007Maxplus,Gondran2008Graphs,Butkovic2010Maxlinear}.

\subsection{Idempotent Semifield}

An idempotent semifield is an algebraic system $(\mathbb{X},\oplus,\otimes,\mathbb{0},\mathbb{1})$, where $\mathbb{X}$ is a non-empty carrier set, $\oplus$ and $\otimes$ are binary operations, called addition and multiplication, $\mathbb{0}$ and $\mathbb{1}$ are distinct elements, called zero and one; such that $(\mathbb{X},\oplus,\mathbb{0})$ is a commutative idempotent monoid, $(\mathbb{X},\otimes,\mathbb{1})$ is an abelian group, multiplication distributes over addition, and $\mathbb{0}$ is absorbing for multiplication.

In the semifield, addition is idempotent, which means the equality $x\oplus x=x$ is valid for each $x\in\mathbb{X}$. The addition induces a partial order relation such that $x\leq y$ if and only if $x\oplus y=y$ for $x,y\in\mathbb{X}$. Note that $\mathbb{0}$ is the least element in terms of this order, and so the inequality $x\ne\mathbb{0}$ implies $x>\mathbb{0}$.

Furthermore, with respect to this partial order, addition exhibits an extremal property in the form of the inequalities $x\oplus y\geq x$ and $x\oplus y\geq y$. Both addition and multiplication are monotone in each argument, which implies that the inequalities $x\leq y$ and $u\leq v$ result in the inequalities $x\oplus u\leq y\oplus v$ and $x\otimes u\leq y\otimes v$. These properties lead, in particular, to the equivalence of the inequality $x\oplus y\leq z$ with the two simultaneous inequalities $x\leq z$ and $y\leq z$.

Multiplication is invertible to allow every non-zero $x\in\mathbb{X}$ to have an inverse $x^{-1}$ such that $x^{-1}\otimes x=\mathbb{1}$. The multiplicative inversion is antitone in the sense that if $x\leq y$ then $x^{-1}\geq y^{-1}$ for all non-zero $x$ and $y$.  

The integer power indicates iterated product defined, for each non-zero $x\ne\mathbb{0}$ and integer $p\geq1$, as $x^{p}=x^{p-1}\otimes x$, $x^{-p}=(x^{-1})^{p}$, $x^{0}=\mathbb{1}$ and $\mathbb{0}^{p}=\mathbb{0}$. We suppose the rational exponents can be defined as well, and take the semifield to be algebraically closed (radicable).

In what follows, the multiplication sign $\otimes$ will be omitted to save writing.

Typical examples of the idempotent semifield under consideration include $\mathbb{R}_{\max,+}=(\mathbb{R}\cup\{-\infty\},\max,+,-\infty,0)$, $\mathbb{R}_{\min,+}=(\mathbb{R}\cup\{+\infty\},\min,+,+\infty,0)$, $\mathbb{R}_{\max,\times}=(\mathbb{R}_{+}\cup\{0\},\max,\times,0,1)$, and $\mathbb{R}_{\min,\times}=(\mathbb{R}_{+}\cup\{+\infty\},\min,\times,+\infty,1)$, where $\mathbb{R}$ denotes the set of real numbers and $\mathbb{R}_{+}=\{x\in\mathbb{R}|x>0\}$. 

Specifically, the semifield $\mathbb{R}_{\max,+}$ is equipped with the maximum operator in the role of addition, and arithmetic addition as multiplication. Zero and one are defined as $-\infty$ and $0$, respectively. For each $x\in\mathbb{R}$, there exists the inverse $x^{-1}$, which is equal to $-x$ in ordinary notation. The power $x^{y}$ can be defined for all $x,y\in\mathbb{R}$ (and thus for rational $y$) to coincide with the arithmetic product $xy$. The partial order induced by addition agrees with the usual linear order on $\mathbb{R}$.

\subsection{Matrix and Vector Algebra}

Consider matrices over the idempotent semifield and denote the set of matrices with $m$ rows and $n$ columns by $\mathbb{X}^{m\times n}$. A matrix with all zero entries is the zero matrix. A matrix is column- (row-) regular if it has no zero columns (rows).

Addition, multiplication, and scalar multiplication of matrices follow the usual rules. For any matrices $\bm{A}=(a_{ij})\in\mathbb{X}^{m\times n}$, $\bm{B}=(b_{ij})\in\mathbb{X}^{m\times n}$ and $\bm{C}=(c_{ij})\in\mathbb{X}^{n\times l}$, and a scalar $x\in\mathbb{X}$, these operations are performed according to the entry-wise formulas
$$
\{\bm{A}\oplus\bm{B}\}_{ij}
=
a_{ij}
\oplus
b_{ij},
\qquad
\{\bm{A}\bm{C}\}_{ij}
=
\bigoplus_{k=1}^{n}a_{ik}c_{kj},
\qquad
\{x\bm{A}\}_{ij}
=
xa_{ij}.
$$

The extremal property of the scalar addition extends to the matrix addition, which implies the entry-wise inequalities $\bm{A}\oplus\bm{B}\geq\bm{A}$ and $\bm{A}\oplus\bm{B}\geq\bm{B}$. All matrix operations are entry-wise monotone in each argument. The inequality $\bm{A}\oplus\bm{B}\leq\bm{C}$ is equivalent to the two inequalities $\bm{A}\leq\bm{C}$ and $\bm{B}\leq\bm{C}$.

Furthermore, we concentrate on square matrices of order $n$ in the set $\mathbb{X}^{n\times n}$. A matrix that has the diagonal entries set to $\mathbb{1}$, and the off-diagonal entries to $\mathbb{0}$ is the identity matrix, which is denoted by $\bm{I}$.

The integer power of a square matrix $\bm{A}$ is routinely defined as $\bm{A}^{0}=\bm{I}$ and $\bm{A}^{p}=\bm{A}^{p-1}\bm{A}=\bm{A}\bm{A}^{p-1}$ for all $p\geq1$.

The trace of a matrix $\bm{A}=(a_{ij})$ is given by
$$
\mathop\mathrm{tr}\bm{A}
=
a_{11}\oplus\cdots\oplus a_{nn}.
$$

A matrix that consists of one column (row) is a column (row) vector. In the following, all vectors are regarded as column vectors, unless otherwise specified. The set of column vectors of length $n$ is denoted by $\mathbb{X}^{n}$. A vector with all zero elements is the zero vector. A vector is called regular if it has no zero components.

Let $\bm{x}=(x_{i})$ be a non-zero vector. The multiplicative conjugate transpose of $\bm{x}$ is a row vector $\bm{x}^{-}=(x_{i}^{-})$, where $x_{i}^{-}=x_{i}^{-1}$ if $x_{i}>\mathbb{0}$, and $x_{i}^{-}=\mathbb{0}$ otherwise.

It follows from the antitone property of the inverse operation that, for regular vectors $\bm{x}$ and $\bm{y}$, the inequality $\bm{x}\leq\bm{y}$ implies that $\bm{x}^{-}\geq\bm{y}^{-}$ and vice versa.

The conjugate transposition exhibits the following properties, which are easy to verify. First, note that $\bm{x}^{-}\bm{x}=\mathbb{1}$ for each non-zero vector $\bm{x}$.

Suppose that $\bm{x},\bm{y}\in\mathbb{X}^{n}$ are regular vectors. Then, the matrix inequality $\bm{x}\bm{y}^{-}\geq(\bm{x}^{-}\bm{y})^{-1}\bm{I}$ holds entry-wise, and becomes $\bm{x}\bm{x}^{-}\geq\bm{I}$ if $\bm{y}=\bm{x}$.

Finally, for any regular vector $\bm{x}\in\mathbb{X}^{n}$, if a matrix $\bm{A}\in\mathbb{X}^{n\times n}$ is row-regular, then $\bm{A}\bm{x}$ is a regular vector. If $\bm{A}$ is column-regular, then $\bm{x}^{-}\bm{A}$ is regular.

\section{Solutions to Linear Inequalities}
\label{S-SLN}

We now present solutions to linear vector inequalities, which form the basis for later investigation of constrained optimization problems. These solutions are often obtained as consequences to the solution of the corresponding equations, and are known under diverse assumptions, at different levels of generality, and in various forms (see, e.g., \cite{Carre1979Graphs,Cuninghamegreen1979Minimax,Zimmermann1981Linear,Baccelli1993Synchronization,Akian2007Maxplus,Butkovic2010Maxlinear}).

In this section we follow the results in \cite{Krivulin2006Solution,Krivulin2009Onsolution,Krivulin2009Methods,Krivulin2013Amultidimensional,Krivulin2013Direct}, which offer a framework to represent the solutions in a compact vector form.

Suppose that, given a matrix $\bm{A}\in\mathbb{X}^{m\times n}$ and a regular vector $\bm{d}\in\mathbb{X}^{m}$, the problem is to find all regular vectors $\bm{x}\in\mathbb{X}^{n}$ that satisfy the inequality
\begin{equation}
\bm{A}\bm{x}
\leq
\bm{d}.
\label{I-Axd}
\end{equation}

The next result offers a solution obtained as a consequence of the solution to the corresponding equation \cite{Krivulin2009Onsolution,Krivulin2009Methods}, and by independent proof \cite{Krivulin2013Direct}. 

\begin{lemma}
\label{L-Axd}
For every column-regular matrix $\bm{A}$ and regular vector $\bm{d}$, all regular solutions to inequality \eqref{I-Axd} are given by
\begin{equation*}
\bm{x}
\leq
(\bm{d}^{-}\bm{A})^{-}.
\label{I-xdA}
\end{equation*}
\end{lemma}

Furthermore, we consider the following problem: given a matrix $\bm{A}\in\mathbb{X}^{n\times n}$ and a vector $\bm{b}\in\mathbb{X}^{n}$, find all regular vectors $\bm{x}\in\mathbb{X}^{n}$ that satisfy the inequality
\begin{equation}
\bm{A}\bm{x}\oplus\bm{b}
\leq
\bm{x}.
\label{I-Axbx}
\end{equation}

To describe a complete solution to the problem, we define a function that maps every matrix $\bm{A}\in\mathbb{X}^{n\times n}$ to a scalar given by
$$
\mathop\mathrm{Tr}(\bm{A})
=
\mathop\mathrm{tr}\bm{A}\oplus\cdots\oplus\mathop\mathrm{tr}\bm{A}^{n}.
$$

We also employ the asterate operator (also known as the Kleene star), which takes $\bm{A}$ to the matrix
$$
\bm{A}^{\ast}
=
\bm{I}\oplus\bm{A}\oplus\cdots\oplus\bm{A}^{n-1}.
$$

Note that the asterate possesses a useful property established by Carr{\'e} \cite{Carre1971Analgebra}. The property states that each matrix $\bm{A}$ with $\mathop\mathrm{Tr}(\bm{A})\leq\mathbb{1}$ satisfies the entry-wise inequality $\bm{A}^{k}\leq\bm{A}^{\ast}$ for all integer $k\geq0$. Specifically, this property makes the equality $\bm{A}^{\ast}\bm{A}^{\ast}=\bm{A}^{\ast}$ valid provided that $\mathop\mathrm{Tr}(\bm{A})\leq\mathbb{1}$.

A direct solution to inequality \eqref{I-Axbx} is given as follows \cite{Krivulin2006Solution,Krivulin2009Onsolution,Krivulin2013Amultidimensional}.
\begin{theorem}\label{T-Axbx}
For every matrix $\bm{A}$ and vector $\bm{b}$, the following statements hold:
\begin{enumerate}
\item If $\mathop\mathrm{Tr}(\bm{A})\leq\mathbb{1}$, then all regular solutions to \eqref{I-Axbx} are given by $\bm{x}=\bm{A}^{\ast}\bm{u}$, where $\bm{u}$ is any regular vector such that $\bm{u}\geq\bm{b}$.
\item If $\mathop\mathrm{Tr}(\bm{A})>\mathbb{1}$, then there is no regular solution.
\end{enumerate}
\end{theorem}

\section{Optimization Problems}
\label{S-OP}

This section is concerned with deriving complete direct solutions to multidimensional constrained optimization problems. The problems consist in minimizing a nonlinear objective function subject to both linear inequality constraints with a matrix and simple boundary constraints. We start with a short overview of the previous results, which provide solutions to problems with reduced sets of constraints. Furthermore, a complete solution to a general problem that involves both constraints is obtained under fairly general assumptions. Two special cases of the solution are discussed which improve the previous results. Finally, we present illustrative examples of two-dimensional optimization problems.

\subsection{Previous Results}

We start with an unconstrained problem that is examined in \cite{Krivulin2011Analgebraic} by applying extremal properties of tropical eigenvalues. Given vectors $\bm{p},\bm{q}\in\mathbb{X}^{n}$, the problem is to find regular vectors $\bm{x}\in\mathbb{X}^{n}$ that
\begin{equation}
\begin{aligned}
&
\text{minimize}
&&
\bm{x}^{-}\bm{p}\oplus\bm{q}^{-}\bm{x}.
\end{aligned}
\label{P-xpqx}
\end{equation}

The problem is reduced to the solving of the eigenvalue-eigenvector problem for a certain matrix. The solution is given by the next statement.

\begin{lemma}\label{L-xpqx}
Let $\bm{p}$ and $\bm{q}$ be regular vectors, and
$$
\theta
=
(\bm{q}^{-}\bm{p})^{1/2}.
$$

Then, the minimum value in problem \eqref{P-xpqx} is equal to $\theta$ and attained at each vector $\bm{x}$ such that
$$
\theta^{-1}\bm{p}
\leq
\bm{x}
\leq
\theta\bm{q}.
$$
\end{lemma}

A different approach based on the solutions to linear inequalities is used in \cite{Krivulin2012Anew,Krivulin2013Direct} to show that the above solution of problem \eqref{P-xpqx} is complete. Moreover, the approach is applied to solve constrained versions of the problem. Specifically, the following problem is considered: given a matrix $\bm{B}\in\mathbb{X}^{n\times n}$, find regular vectors $\bm{x}$ that
\begin{equation}
\begin{aligned}
&
\text{minimize}
&&
\bm{x}^{-}\bm{p}\oplus\bm{q}^{-}\bm{x},
\\
&
\text{subject to}
&&
\bm{B}\bm{x}
\leq
\bm{x}.
\end{aligned}
\label{P-xpqxBxx}
\end{equation}

The solution, which is given in \cite{Krivulin2012Anew} under some restrictive assumptions on the matrix $\bm{B}$, can readily be extended to arbitrary matrices by using the result of Theorem~\ref{T-Axbx}, and then written in the following form.
\begin{theorem}
Let $\bm{B}$ be a matrix with $\mathop\mathrm{Tr}(\bm{B})\leq\mathbb{1}$, $\bm{p}$ and $\bm{q}$ regular vectors, and
\begin{equation}
\theta
=
((\bm{B}^{\ast}(\bm{q}^{-}\bm{B}^{\ast})^{-})^{-}\bm{p})^{1/2}.
\label{E-theta0}
\end{equation}

Then, the minimum value in problem \eqref{P-xpqxBxx} is equal to $\theta$ and attained at
$$
\bm{x}
=
\theta\bm{B}^{\ast}(\bm{q}^{-}\bm{B}^{\ast})^{-}.
$$
\end{theorem}

Note that the theorem offers a particular solution to the problem rather than provides a complete solution.

Furthermore, given vectors $\bm{g},\bm{h}\in\mathbb{X}^{n}$, consider a problem with two-sided boundary constraints to find regular vectors $\bm{x}$ that
\begin{equation}
\begin{aligned}
&
\text{minimize}
&&
\bm{x}^{-}\bm{p}\oplus\bm{q}^{-}\bm{x},
\\
&
\text{subject to}
&&
\bm{g}
\leq
\bm{x}
\leq
\bm{h}.
\end{aligned}
\label{P-xpqxgxh}
\end{equation}

The complete solution obtained in \cite{Krivulin2013Direct} is as follows.
\begin{theorem}
\label{T-xpqxgxh}
Let $\bm{p}$, $\bm{q}$, $\bm{g}$, and $\bm{h}$ be regular vectors such that $\bm{g}\leq\bm{h}$, and
\begin{equation*}
\theta
=
(\bm{q}^{-}\bm{p})^{1/2}\oplus\bm{h}^{-}\bm{p}\oplus\bm{q}^{-}\bm{g}.
\label{E-theta1}
\end{equation*}

Then, the minimum in problem \eqref{P-xpqxgxh} is equal to $\theta$ and all regular solutions of the problem are given by the condition
$$
\bm{g}\oplus\theta^{-1}\bm{p}
\leq
\bm{x}
\leq
(\bm{h}^{-}\oplus\theta^{-1}\bm{q}^{-})^{-}.
$$
\end{theorem}

Below, we examine a new general problem, which combines the constraints in problems \eqref{P-xpqxBxx} and \eqref{P-xpqxgxh}, and includes both these problems as special cases.

\subsection{New Optimization Problem with Combined Constraints}

We now are in a position to formulate and solve a new constrained optimization problem. The solution follows the approach developed in \cite{Krivulin2013Amultidimensional,Krivulin2014Aconstrained}, which is based on the introduction of an auxiliary variable and the reduction of the problem to the solution of a parametrized system of linear inequalities, where the new variable plays the role of a parameter. The existence condition for the solution of the system is used to evaluate the parameter, whereas the complete solution to the system is taken as the solution to the optimization problem.

Given vectors $\bm{p},\bm{q},\bm{g},\bm{h}\in\mathbb{X}^{n}$, and a matrix $\bm{B}\in\mathbb{X}^{n\times n}$, consider the problem to find all regular vectors $\bm{x}\in\mathbb{X}^{n}$ that
\begin{equation}
\begin{aligned}
&
\text{minimize}
&&
\bm{x}^{-}\bm{p}\oplus\bm{q}^{-}\bm{x},
\\
&
\text{subject to}
&&
\bm{B}\bm{x}\oplus\bm{g}
\leq
\bm{x},
\\
&
&&
\bm{x}
\leq
\bm{h}.
\end{aligned}
\label{P-xpqxBxxgxh}
\end{equation}

The constraints in the problem can also be written in the equivalent form
\begin{align*}
&\bm{B}\bm{x}
\leq
\bm{x},
\\
&
\bm{g}
\leq
\bm{x}
\leq
\bm{h}.
\end{align*}

The next statement gives a complete direct solution to the problem.
\begin{theorem}
\label{T-xpqxBxxgxh}
Let $\bm{B}$ be a matrix with $\mathop\mathrm{Tr}(\bm{B})\leq\mathbb{1}$, $\bm{p}$ be a non-zero vector, $\bm{q}$ and $\bm{h}$ regular vectors, and $\bm{g}$ a vector such that $\bm{h}^{-}\bm{B}^{\ast}\bm{g}\leq\mathbb{1}$. Define a scalar
\begin{equation}
\theta
=
(\bm{q}^{-}\bm{B}^{\ast}\bm{p})^{1/2}
\oplus
\bm{h}^{-}\bm{B}^{\ast}\bm{p}\oplus\bm{q}^{-}\bm{B}^{\ast}\bm{g}.
\label{E-theta}
\end{equation}

Then, the minimum value in problem \eqref{P-xpqxBxxgxh} is equal to $\theta$ and all regular solutions of the problem are given by
$$
\bm{x}
=
\bm{B}^{\ast}\bm{u},
$$
where $\bm{u}$ is any regular vector such that
\begin{equation}
\bm{g}\oplus\theta^{-1}\bm{p}
\leq
\bm{u}
\leq
((\bm{h}^{-}\oplus\theta^{-1}\bm{q}^{-})\bm{B}^{\ast})^{-}.
\label{E-xBu}
\end{equation}
\end{theorem}
\begin{proof}
Suppose that $\theta$ is the minimum of the objective function in problem \eqref{P-xpqxBxxgxh} over all regular $\bm{x}$, and note that $\theta\geq(\bm{q}^{-}\bm{B}^{\ast}\bm{p})^{1/2}\geq(\bm{q}^{-}\bm{p})^{1/2}>\mathbb{0}$. Then, all solutions to the problem are given by the system
\begin{align*}
\bm{x}^{-}\bm{p}\oplus\bm{q}^{-}\bm{x}
&=
\theta,
\\
\bm{B}\bm{x}\oplus\bm{g}
&\leq
\bm{x},
\\
\bm{x}
&\leq
\bm{h}.
\end{align*}

Since $\theta$ is the minimum of the objective function, the solution set remains unchanged if we replace the first equation by the inequality $\bm{x}^{-}\bm{p}\oplus\bm{q}^{-}\bm{x}\leq\theta$ and then substitute this inequality with equivalent two inequalities as follows
\begin{align*}
\bm{x}^{-}\bm{p}
&\leq
\theta,
\\
\bm{q}^{-}\bm{x}
&\leq
\theta,
\\
\bm{B}\bm{x}\oplus\bm{g}
&\leq
\bm{x},
\\
\bm{x}
&\leq
\bm{h}.
\end{align*}

After the application of Lemma~\ref{L-Axd} to the first two inequalities, the system becomes
\begin{align*}
\theta^{-1}\bm{p}
&\leq
\bm{x},
\\
\bm{x}
&\leq
\theta\bm{q},
\\
\bm{B}\bm{x}\oplus\bm{g}
&\leq
\bm{x},
\\
\bm{x}
&\leq
\bm{h}.
\end{align*}

We now combine the inequalities in the system as follows. The first and third inequalities are equivalent to the inequality $\bm{B}\bm{x}\oplus\bm{g}\oplus\theta^{-1}\bm{p}\leq\bm{x}$.

The other two inequalities are replaced by $\bm{x}^{-}\geq\theta^{-1}\bm{q}^{-}$ and $\bm{x}^{-}\geq\bm{h}^{-}$, which are equivalent to $\bm{x}^{-}\geq\bm{h}^{-}\oplus\theta^{-1}\bm{q}^{-}$, and thus to $\bm{x}\leq(\bm{h}^{-}\oplus\theta^{-1}\bm{q}^{-})^{-}$.

After the rearrangement of the system, we arrive at the double inequality
$$
\bm{B}\bm{x}\oplus\bm{g}
\oplus
\theta^{-1}\bm{p}
\leq
\bm{x}
\leq
(\bm{h}^{-}
\oplus
\theta^{-1}\bm{q}^{-})^{-}.
$$

The solution of the left inequality by using Theorem~\ref{T-Axbx} gives the result
$$
\bm{x}
=
\bm{B}^{\ast}\bm{u},
\qquad
\bm{u}
\geq
\bm{g}\oplus\theta^{-1}\bm{p}.
$$

Substitution of this solution into the right inequality yields the inequality
$$
\bm{B}^{\ast}\bm{u}
\leq
(\bm{h}^{-}
\oplus
\theta^{-1}\bm{q}^{-})^{-},
$$
which, by Lemma~\ref{L-Axd}, has the solution
$$
\bm{u}
\leq
((\bm{h}^{-}
\oplus
\theta^{-1}\bm{q}^{-})\bm{B}^{\ast})^{-}.
$$

By coupling both lower and upper bounds on $\bm{u}$, we arrive at the solution in the form of \eqref{E-xBu}. The solution set defined by \eqref{E-xBu} is non-empty if and only if
$$
\bm{g}\oplus\theta^{-1}\bm{p}
\leq
((\bm{h}^{-}
\oplus
\theta^{-1}\bm{q}^{-})\bm{B}^{\ast})^{-}.
$$

The left multiplication of this inequality by $(\bm{h}^{-}\oplus\theta^{-1}\bm{q}^{-})\bm{B}^{\ast}$ and application of one property of conjugate transposition lead to 
$$
(\bm{h}^{-}\oplus\theta^{-1}\bm{q}^{-})\bm{B}^{\ast}(\bm{g}\oplus\theta^{-1}\bm{p})
\leq
(\bm{h}^{-}\oplus\theta^{-1}\bm{q}^{-})\bm{B}^{\ast}((\bm{h}^{-}\oplus\theta^{-1}\bm{q}^{-})\bm{B}^{\ast})^{-}
=
\mathbb{1},
$$
which results in the new inequality
$$
(\bm{h}^{-}
\oplus
\theta^{-1}\bm{q}^{-})\bm{B}^{\ast}
(\bm{g}\oplus\theta^{-1}\bm{p})
\leq
\mathbb{1}.
$$

Since the left multiplication of the latter inequality by $((\bm{h}^{-}\oplus\theta^{-1}\bm{q}^{-})\bm{B}^{\ast})^{-}$ and the other property of conjugate transposition give the former inequality, both inequalities are equivalent. The obtained inequality can further be rewritten as
$$
\theta^{-2}\bm{q}^{-}\bm{B}^{\ast}\bm{p}
\oplus
\theta^{-1}(\bm{h}^{-}\bm{B}^{\ast}\bm{p}
\oplus
\bm{q}^{-}\bm{B}^{\ast}\bm{g})
\oplus
\bm{h}^{-}\bm{B}^{\ast}\bm{g}
\leq
\mathbb{1},
$$
and then represented by the equivalent system
\begin{align*}
\theta^{-2}\bm{q}^{-}\bm{B}^{\ast}\bm{p}
&\leq
\mathbb{1},
\\
\theta^{-1}(\bm{h}^{-}\bm{B}^{\ast}\bm{p}
\oplus
\bm{q}^{-}\bm{B}^{\ast}\bm{g})
&\leq
\mathbb{1},
\\
\bm{h}^{-}\bm{B}^{\ast}\bm{g}
&\leq
\mathbb{1}.
\end{align*}

Note that the third inequality in the system is valid by the condition of the theorem. After rearrangement of terms, the first two inequalities become
\begin{align*}
\theta
&\geq
(\bm{q}^{-}\bm{B}^{\ast}\bm{p})^{1/2},
\\
\theta
&\geq
\bm{h}^{-}\bm{B}^{\ast}\bm{p}\oplus\bm{q}^{-}\bm{B}^{\ast}\bm{g},
\end{align*}
and then finally lead to one inequality
$$
\theta
\geq
(\bm{q}^{-}\bm{B}^{\ast}\bm{p})^{1/2}
\oplus
\bm{h}^{-}\bm{B}^{\ast}\bm{p}\oplus\bm{q}^{-}\bm{B}^{\ast}\bm{g}.
$$

Since $\theta$ is assumed to be the minimum value of the objective function, the last inequality has to be satisfied as an equality, which gives \eqref{E-theta}.
\qed
\end{proof}

\subsection{Particular Cases}

We now examine particular cases, in which the feasible solution set is defined either by a linear inequality with a matrix or by two-sided boundary constraints.

First, we offer a new complete solution to problem \eqref{P-xpqxBxx}, which does not have the boundary constraints. A slight modification to the proof of Theorem~\ref{T-xpqxBxxgxh} yields the solution in the following form.
\begin{corollary}
\label{C-xpqxBxx}
Let $\bm{B}$ be a matrix with $\mathop\mathrm{Tr}(\bm{B})\leq\mathbb{1}$, $\bm{p}$ be a non-zero vector, and $\bm{q}$ a regular vector. Define a scalar
\begin{equation}
\theta
=
(\bm{q}^{-}\bm{B}^{\ast}\bm{p})^{1/2}.
\label{E-theta2}
\end{equation}

Then, the minimum in \eqref{P-xpqxBxx} is $\theta$ and all regular solutions are given by
\begin{equation*}
\bm{x}
=
\bm{B}^{\ast}\bm{u},
\qquad
\theta^{-1}\bm{p}
\leq
\bm{u}
\leq
\theta(\bm{q}^{-}\bm{B}^{\ast})^{-}.
\end{equation*}
\end{corollary}

Although the expression at \eqref{E-theta2} offers the minimum in a different and more compact form than that at \eqref{E-theta0}, both representations prove to be equivalent.

To verify that these representations coincide, we first note that $\bm{B}^{\ast}\bm{B}^{\ast}=\bm{B}^{\ast}$ and then apply the properties of conjugate transposition to write
$$
\bm{B}^{\ast}(\bm{q}^{-}\bm{B}^{\ast})^{-}
=
\bm{B}^{\ast}(\bm{q}^{-}\bm{B}^{\ast}\bm{B}^{\ast})^{-}
\leq
(\bm{q}^{-}\bm{B}^{\ast})^{-}\bm{q}^{-}\bm{B}^{\ast}\bm{B}^{\ast}(\bm{q}^{-}\bm{B}^{\ast}\bm{B}^{\ast})^{-}
=
(\bm{q}^{-}\bm{B}^{\ast})^{-},
$$
which implies that the inequality $\bm{B}^{\ast}(\bm{q}^{-}\bm{B}^{\ast})^{-}\leq(\bm{q}^{-}\bm{B}^{\ast})^{-}$ holds.

Since $\bm{B}^{\ast}\geq\bm{I}$, the opposite inequality $\bm{B}^{\ast}(\bm{q}^{-}\bm{B}^{\ast})^{-}\geq(\bm{q}^{-}\bm{B}^{\ast})^{-}$ is valid as well. Both inequalities result in the equality $\bm{B}^{\ast}(\bm{q}^{-}\bm{B}^{\ast})^{-}=(\bm{q}^{-}\bm{B}^{\ast})^{-}$, and thus in the equality $(\bm{B}^{\ast}(\bm{q}^{-}\bm{B}^{\ast})^{-})^{-}=\bm{q}^{-}\bm{B}^{\ast}$. Finally, the right multiplication by $\bm{p}$ and extraction of square roots lead to the desired result.

Furthermore, we put $\bm{B}$ to be the zero matrix in \eqref{P-xpqxBxxgxh} and so arrive at problem \eqref{P-xpqxgxh}, which can be completely solved through a direct consequence of Theorem~\ref{T-xpqxBxxgxh}. Clearly, the new solution of \eqref{P-xpqxgxh} coincides with that given by Theorem~\ref{T-xpqxgxh}, and even involves somewhat less assumptions on the vectors under consideration.

\subsection{Numerical Examples and Graphical Illustration}

To illustrate the results obtained above, we present examples of two-dimensional problems in the setting of the idempotent semifield $\mathbb{R}_{\max,+}$ and provide geometric interpretation on the plane with a Cartesian coordinate system.

Consider problem \eqref{P-xpqxBxxgxh} formulated in terms of $\mathbb{R}_{\max,+}$ under the assumptions that
$$
\bm{p}
=
\left(
\begin{array}{c}
3
\\
14
\end{array}
\right),
\quad
\bm{q}
=
\left(
\begin{array}{r}
-12
\\
-4
\end{array}
\right),
\quad
\bm{g}
=
\left(
\begin{array}{r}
2
\\
-8
\end{array}
\right),
\quad
\bm{h}
=
\left(
\begin{array}{c}
6
\\
8
\end{array}
\right),
\quad
\bm{B}
=
\left(
\begin{array}{rr}
0 & -4
\\
-8 & -6
\end{array}
\right).
$$

Prior to solving the general problem, we examine several special cases.

We start with problem \eqref{P-xpqx} without constraints, which has a complete solution given by a consequence of Theorem~\ref{T-xpqxBxxgxh} (see also Lemma~\ref{L-xpqx}). According to this result, the minimum in the unconstrained problem is given by
$$
\theta_{1}
=
(\bm{q}^{-}\bm{p})^{1/2}
=
9,
$$ 
and attained if and only if the vector $\bm{x}$ satisfies the conditions
$$
\bm{x}_{1}^{\prime}
\leq
\bm{x}
\leq
\bm{x}_{1}^{\prime\prime},
\qquad
\bm{x}_{1}^{\prime}
=
\theta_{1}^{-1}\bm{p}
=
\left(
\begin{array}{r}
-6
\\
5
\end{array}
\right),
\quad
\bm{x}_{1}^{\prime\prime}
=
\theta_{1}\bm{q}
=
\left(
\begin{array}{r}
-3
\\
5
\end{array}
\right).
$$

A graphical illustration of the result is given in Fig.~\ref{F-SUPPBC} (left), where the solutions form a horizontal segment between the ends of the vectors $\bm{x}_{1}^{\prime}$ and $\bm{x}_{1}^{\prime\prime}$.
\begin{figure}
\setlength{\unitlength}{1mm}
\begin{center}
\begin{picture}(50,55)

\put(0,20){\vector(1,0){50}}
\put(25,0){\vector(0,1){55}}


\put(25,20){\thicklines\line(1,5){5.5}}
\put(30.8,48.5){\thicklines\vector(1,4){0}}

\put(25,20){\thicklines\vector(-3,-1){24}}

\put(1,12){\line(1,0){30}}
\put(1,12){\line(0,1){36}}

\put(31,48){\line(-1,0){30}}
\put(31,48){\line(0,-1){36}}

\put(25,20){\thicklines\line(-3,5){6}}
\put(18.8,30){\thicklines\vector(-3,4){0}}

\put(25,20){\thicklines\line(-6,5){12}}
\put(13,30){\thicklines\vector(-1,1){0}}

\put(31,48){\line(-1,-1){18}}
\put(1,12){\line(1,1){18}}

\put(13,30){\line(1,0){6}}
\put(13,30.15){\thicklines\line(1,0){6}}
\put(13,29.85){\thicklines\line(1,0){6}}

\put(0,8){$\bm{q}$}
\put(33,48){$\bm{p}$}

\put(9,32){$\bm{x}_{1}^{\prime}$}
\put(20,32){$\bm{x}_{1}^{\prime\prime}$}


\end{picture}
\hspace{15\unitlength}
\begin{picture}(50,55)

\put(0,20){\vector(1,0){50}}
\put(25,0){\vector(0,1){55}}


\put(1,12){\line(1,0){30}}
\put(1,12){\line(0,1){36}}

\put(31,48){\line(-1,0){30}}
\put(31,48){\line(0,-1){36}}

\put(0,8){$\bm{q}$}
\put(33,48){$\bm{p}$}


\put(25,20){\thicklines\vector(3,4){12}}

\put(37,36){\thicklines\line(-1,0){8}}
\put(37,36){\thicklines\line(0,-1){32}}

\multiput(37,36)(-1,0){9}{\line(0,-1){1}}
\multiput(37,36)(0,-1){32}{\line(-1,0){1}}

\put(25,20){\thicklines\vector(1,-4){4}}
\put(29,4){\thicklines\line(1,0){8}}
\put(29,4){\thicklines\line(0,1){32}}

\multiput(29,4)(1,0){9}{\line(0,1){1}}
\multiput(29,4)(0,1){32}{\line(1,0){1}}

\put(27,1){$\bm{g}$}
\put(39,36){$\bm{h}$}

\put(25,20){\thicklines\vector(1,0){4}}
\put(25,20){\thicklines\vector(1,4){4}}

\put(28.8,20){\thicklines\line(0,1){16}}
\put(29.2,20){\thicklines\line(0,1){16}}

\put(32,22){$\bm{x}_{2}^{\prime}$}
\put(26,38){$\bm{x}_{2}^{\prime\prime}$}

\end{picture}
\end{center}
\caption{Solutions to problems without constraints (left) and with two-sided boundary constraints (right).}
\label{F-SUPPBC}
\end{figure}
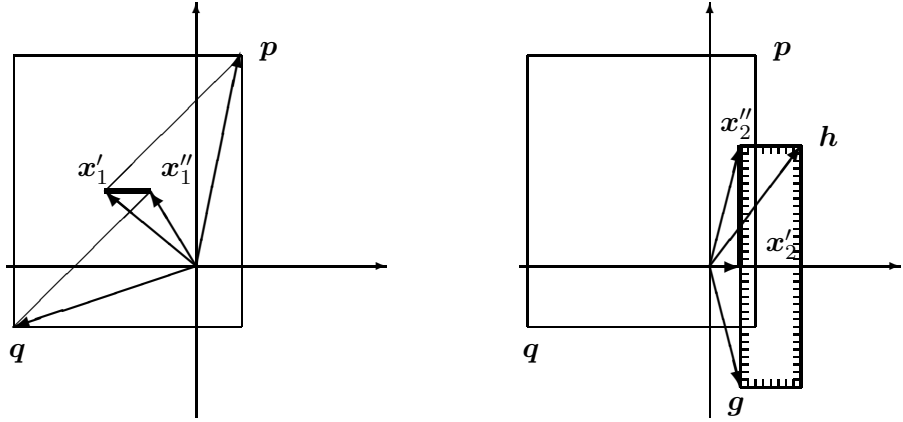

Furthermore, we consider the problem in the form \eqref{P-xpqxgxh} with two-sided boundary constraints $\bm{g}\leq\bm{x}\leq\bm{h}$. It follows from Theorem~\ref{T-xpqxgxh} (or as another consequence of Theorem~\ref{T-xpqxBxxgxh}) that the minimum in the problem is calculated as
$$
\theta_{2}
=
(\bm{q}^{-}\bm{p})^{1/2}
\oplus
\bm{h}^{-}\bm{p}
\oplus
\bm{q}^{-}\bm{g}
=
14.
$$

The solution set consists of those vectors $\bm{x}$ that satisfy the double inequality
$$
\bm{x}_{2}^{\prime}
\leq
\bm{x}
\leq
\bm{x}_{2}^{\prime\prime},
\qquad
\bm{x}_{2}^{\prime}
=
\bm{g}\oplus\theta_{2}^{-1}\bm{p}
=
\left(
\begin{array}{c}
2
\\
0
\end{array}
\right),
\quad
\bm{x}_{2}^{\prime\prime}
=
(\bm{h}^{-}\oplus\theta_{2}^{-1}\bm{q}^{-})^{-}
=
\left(
\begin{array}{c}
2
\\
8
\end{array}
\right).
$$

The solutions of the problem are indicated on Fig.~\ref{F-SUPPBC} (right) by a thick vertical segment on the left side of the rectangle that represents the feasible set.

We now examine problem \eqref{P-xpqxBxx} with the linear inequality constraints $\bm{B}\bm{x}\leq\bm{x}$. We calculate 
$$
\bm{B}^{\ast}
=
\bm{I}\oplus\bm{B}
=
\left(
\begin{array}{rr}
0 & -4
\\
-8 & 0
\end{array}
\right),
\qquad
\bm{q}^{-}\bm{B}^{\ast}
=
\left(
\begin{array}{cc}
12
&
8
\end{array}
\right).
$$

The application of Corollary~\ref{C-xpqxBxx} gives the minimum value
$$
\theta_{3}
=
(\bm{q}^{-}\bm{B}^{\ast}\bm{p})^{1/2}
=
11,
$$
which is attained if and only if $\bm{x}=\bm{B}^{\ast}\bm{u}$ for all $\bm{u}$ such that
$$
\bm{u}_{3}^{\prime}
\leq
\bm{u}
\leq
\bm{u}_{3}^{\prime\prime},
\qquad
\bm{u}_{3}^{\prime}
=
\theta^{-1}\bm{p}
=
\left(
\begin{array}{r}
-8
\\
3
\end{array}
\right),
\quad
\bm{u}_{3}^{\prime\prime}
=
\theta(\bm{q}^{-}\bm{B}^{\ast})^{-}
=
\left(
\begin{array}{r}
-1
\\
3
\end{array}
\right).
$$

After multiplication of $\bm{B}^{\ast}$ by both bounds on $\bm{u}$, we conclude that the problem has the unique solution
$$
\bm{x}_{3}
=
\bm{B}^{\ast}\bm{u}_{3}^{\prime}
=
\bm{B}^{\ast}\bm{u}_{3}^{\prime\prime}
=
\left(
\begin{array}{r}
-1
\\
3
\end{array}
\right).
$$

Figure~\ref{F-SPLICLIBC} (left) shows the solution point located on the upper side of the strip, which represents the solution of the inequality $\bm{B}\bm{x}\leq\bm{x}$. The columns of the matrices $\bm{B}=(\bm{b}_{1},\bm{b}_{2})$ and $\bm{B}^{\ast}=(\bm{b}_{1}^{\ast},\bm{b}_{2}^{\ast})$ are also included.
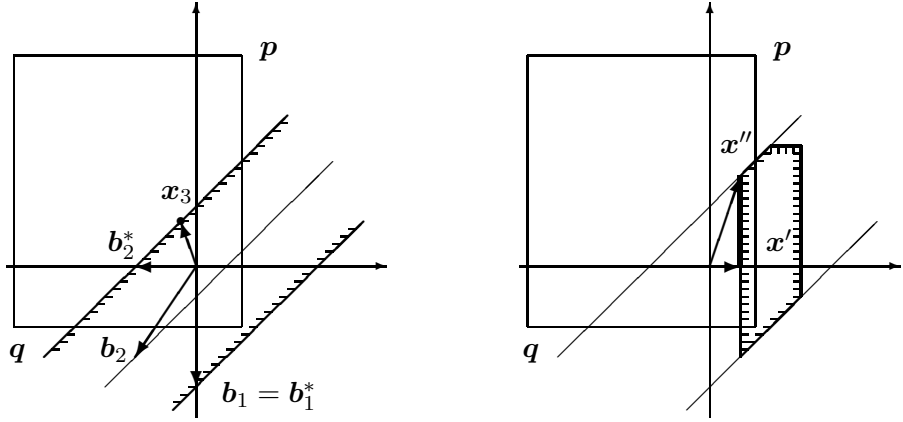
\begin{figure}
\setlength{\unitlength}{1mm}
\begin{center}
\begin{picture}(50,55)

\put(0,20){\vector(1,0){50}}
\put(25,0){\vector(0,1){55}}


\put(1,12){\line(1,0){30}}
\put(1,12){\line(0,1){36}}

\put(31,48){\line(-1,0){30}}
\put(31,48){\line(0,-1){36}}

\put(0,8){$\bm{q}$}
\put(33,48){$\bm{p}$}


\put(25,20){\thicklines\vector(-2,-3){8}}
\put(5,8){\thicklines\line(1,1){32}}
\multiput(6,9)(1,1){31}{\line(1,0){1}}

\put(25,20){\thicklines\vector(-1,0){8}}
\put(13,4){\line(1,1){30}}

\put(25,20){\thicklines\vector(0,-1){16}}
\put(22,1){\thicklines\line(1,1){25}}
\multiput(23,2)(1,1){24}{\line(-1,0){1}}

\put(13,22){$\bm{b}_{2}^{\ast}$}

\put(12,8){$\bm{b}_{2}$}

\put(28,2){$\bm{b}_{1}=\bm{b}_{1}^{\ast}$}

\put(23,26){\circle*{1}}

\put(25,20){\thicklines\vector(-1,3){2}}

\put(20,29){$\bm{x}_{3}$}

\end{picture}
\hspace{15\unitlength}
\begin{picture}(50,55)

\put(0,20){\vector(1,0){50}}
\put(25,0){\vector(0,1){55}}


\put(1,12){\line(1,0){30}}
\put(1,12){\line(0,1){36}}

\put(31,48){\line(-1,0){30}}
\put(31,48){\line(0,-1){36}}

\put(0,8){$\bm{q}$}
\put(33,48){$\bm{p}$}


\put(37,36){\thicklines\line(-1,0){4}}
\put(37,36){\thicklines\line(0,-1){20}}

\multiput(37,36)(-1,0){5}{\line(0,-1){1}}
\multiput(37,36)(0,-1){20}{\line(-1,0){1}}

\put(29,8){\thicklines\line(0,1){24}}

\multiput(29,9)(0,1){24}{\line(1,0){1}}

\put(25,20){\thicklines\vector(1,0){4}}
\put(25,20){\thicklines\vector(1,3){4}}

\put(28.8,20){\thicklines\line(0,1){12}}
\put(29.2,20){\thicklines\line(0,1){12}}

\put(32,22){$\bm{x}^{\prime}$}
\put(26,35){$\bm{x}^{\prime\prime}$}


\put(5,8){\line(1,1){32}}
\put(29,32){\thicklines\line(1,1){4}}
\multiput(30,33)(1,1){3}{\line(1,0){1}}

\put(22,1){\line(1,1){25}}
\put(29,8){\thicklines\line(1,1){8}}
\multiput(30,9)(1,1){8}{\line(-1,0){1}}

\end{picture}
\end{center}
\caption{Solutions to problems with linear inequality constraints (left) and with both linear inequality and two-sided boundary constraints (right).}
\label{F-SPLICLIBC}
\end{figure}

Finally, we consider general problem \eqref{P-xpqxBxxgxh}. To solve the problem, we calculate
$$
\bm{B}^{\ast}\bm{p}
=
\left(
\begin{array}{c}
10
\\
14
\end{array}
\right),
\qquad
\bm{h}^{-}\bm{B}^{\ast}\bm{p}
=
6,
\qquad
\bm{q}^{-}\bm{B}^{\ast}\bm{g}
=
14.
$$

It follows from Theorem~\ref{T-xpqxBxxgxh} that the minimum in the problem is given by
$$
\theta
=
(\bm{q}^{-}\bm{B}^{\ast}\bm{p})^{1/2}
\oplus
\bm{h}^{-}\bm{B}^{\ast}\bm{p}\oplus\bm{q}^{-}\bm{B}^{\ast}\bm{g}
=
14.
$$

This minimum is attained only at $\bm{x}=\bm{B}^{\ast}\bm{u}$, where $\bm{u}$ is any vector such that
$$
\bm{u}^{\prime}
\leq
\bm{u}
\leq
\bm{u}^{\prime\prime},
\qquad
\bm{u}^{\prime}
=
\bm{g}\oplus\theta^{-1}\bm{p}
=
\left(
\begin{array}{c}
2
\\
0
\end{array}
\right),
\quad
\bm{u}^{\prime\prime}
=
((\bm{h}^{-}\oplus\theta^{-1}\bm{q}^{-})\bm{B}^{\ast})^{-}
=
\left(
\begin{array}{c}
2
\\
6
\end{array}
\right).
$$
 
Turning to the solution of the problem, we arrive at the set of vectors $\bm{x}$ that satisfy the conditions
$$
\bm{x}^{\prime}
\leq
\bm{x}
\leq
\bm{x}^{\prime\prime},
\qquad
\bm{x}^{\prime}
=
\bm{B}^{\ast}\bm{u}^{\prime}
=
\left(
\begin{array}{c}
2
\\
0
\end{array}
\right),
\quad
\bm{x}^{\prime\prime}
=
\bm{B}^{\ast}\bm{u}^{\prime\prime}
=
\left(
\begin{array}{cc}
2
\\
6
\end{array}
\right).
$$

The solution is shown on Fig.~\ref{F-SPLICLIBC} (right) by the thick vertical segment on the left side of the polygon which describes the feasible set.

\section{Application to Location Analysis}
\label{S-ALA}

In this section, we apply the above results to solve minimax single facility location problems, which are often called the Rawls problems \cite{Hansen1980Location,Hansen1981Constrained}, but also known as Messenger Boy problems \cite{Elzinga1972Geometrical} and $1$-center problems \cite{Drezner2011Continuous}. We consider a new constrained problem on a multidimensional space with Chebyshev distance. A complete direct solution is obtained which extends the results in \cite{Krivulin2011Analgebraic,Krivulin2012Anew,Krivulin2013Direct} by taking into account a more general system of constraints.

Let $\bm{r}=(r_{i})$ and $\bm{s}=(s_{i})$ be vectors in $\mathbb{R}^{n}$. The Chebyshev distance ($L_{\infty}$, maximum, dominance, lattice, king-move, or chessboard metric) between the vectors is calculated as
\begin{equation}
\rho(\bm{r},\bm{s})
=
\max_{1\leq i\leq n}|r_{i}-s_{i}|.
\label{E-Chebyshev}
\end{equation}

Consider the following Chebyshev single facility location problem. Given $m$ vectors $\bm{r}_{j}=(r_{ij})\in\mathbb{R}^{n}$ and constants $w_{j}\in\mathbb{R}$ for each $j=1,\ldots,m$, a matrix $\bm{B}=(b_{ij})\in\mathbb{R}^{n\times n}$, and vectors $\bm{g}=(g_{i})\in\mathbb{R}^{n}$, $\bm{h}=(h_{i})\in\mathbb{R}^{n}$, the problem is to find the vectors $\bm{x}=(x_{i})\in\mathbb{R}^{n}$ that  
\begin{equation}
\begin{aligned}
&
\text{minimize}
&&
\max_{1\leq j\leq m}\left(\max_{1\leq i\leq n}|r_{ij}-x_{i}|+w_{i}\right),
\\
&
\text{subject to}
&&
x_{j}+b_{ij}
\leq
x_{i},
\\
&&&
g_{i}
\leq
x_{i}
\leq
h_{i},
\qquad
j=1,\ldots,n,
\quad
i=1,\ldots,n.
\end{aligned}
\label{P-Chebyshev}
\end{equation}

Note that the feasible location area is formed in $\mathbb{R}^{n}$ by the intersection of the hyper-rectangle defined by the boundary constraints with closed half-spaces given by the other inequalities. 

To solve the problem, we represent it in terms of the semifield $\mathbb{R}_{\max,+}$. First, we put \eqref{E-Chebyshev} in the equivalent form
$$
\rho(\bm{r},\bm{s})
=
\bigoplus_{i=1}^{n}(s_{i}^{-1}r_{i}\oplus r_{i}^{-1}s_{i})
=
\bm{s}^{-}\bm{r}\oplus\bm{r}^{-}\bm{s}.
$$

Furthermore, we define the vectors
$$
\bm{p}
=
w_{1}\bm{r}_{1}\oplus\cdots\oplus w_{m}\bm{r}_{m},
\qquad
\bm{q}^{-}
=
w_{1}\bm{r}_{1}^{-}\oplus\cdots\oplus w_{m}\bm{r}_{m}^{-}.
$$

The objective function in problem \eqref{P-Chebyshev} becomes
$$
\bigoplus_{i=1}^{m}w_{i}\rho(\bm{r}_{i},\bm{x})
=
\bigoplus_{i=1}^{m}w_{i}(\bm{x}^{-}\bm{r}_{i}\oplus\bm{r}_{i}^{-}\bm{x})
=
\bm{x}^{-}\bm{p}
\oplus
\bm{q}^{-}\bm{x}.
$$

We now combine the constraints $x_{j}+b_{ij}\leq x_{i}$ for all $j=1,\ldots,n$ into one inequality for each $i$, and write the obtained inequalities in terms of $\mathbb{R}_{\max,+}$ as
$$
b_{i1}x_{1}\oplus\cdots\oplus b_{in}x_{n}
\leq
x_{i},
\qquad
i=1,\ldots,n.
$$

After rewriting the above inequalities and the boundary constraints in matrix-vector form, we obtain the problem in the form \eqref{P-xpqxBxxgxh}, where all given vectors have real components. Since these vectors are clearly regular in the sense of $\mathbb{R}_{\max,+}$, they satisfy the conditions of Theorem~\ref{T-xpqxBxxgxh}, which completely solves the problem.

As an illustration, consider the two-dimensional problem with given points
$$
\bm{r}_{1}
=
\left(
\begin{array}{r}
-7
\\
12
\end{array}
\right),
\quad
\bm{r}_{2}
=
\left(
\begin{array}{c}
2
\\
10
\end{array}
\right),
\quad
\bm{r}_{3}
=
\left(
\begin{array}{r}
-10
\\
3
\end{array}
\right),
\quad
\bm{r}_{4}
=
\left(
\begin{array}{r}
-4
\\
4
\end{array}
\right),
\quad
\bm{r}_{5}
=
\left(
\begin{array}{r}
-4
\\
-3
\end{array}
\right),
$$
and constants $w_{1}=w_{3}=2$, $w_{2}=w_{4}=w_{5}=1$. For the sake of simplicity, we take the same matrix $\bm{B}$ and vectors $\bm{g},\bm{h}$ as in the examples considered above.

To reduce the location problem to problem \eqref{P-xpqxBxxgxh}, we first calculate the vectors
$$
\bm{p}
=
\left(
\begin{array}{r}
3
\\
14
\end{array}
\right),
\qquad
\bm{q}
=
\left(
\begin{array}{rr}
-12
\\
-4
\end{array}
\right).
$$

These vectors define two opposite corners of the minimum rectangle which encloses all points $w_{i}\bm{r}_{i}$ and $w_{i}^{-1}\bm{r}_{i}$. The rectangle is depicted in Fig.~\ref{F-MERSLP} (left).
\begin{figure}
\setlength{\unitlength}{1mm}
\begin{center}
\begin{picture}(50,55)

\put(0,20){\vector(1,0){50}}
\put(25,0){\vector(0,1){55}}


\put(1,12){\line(1,0){30}}
\put(1,12){\line(0,1){36}}

\put(31,48){\line(-1,0){30}}
\put(31,48){\line(0,-1){36}}

\put(0,8){$\bm{q}$}
\put(33,48){$\bm{p}$}

\put(15,48){\circle{1}}
\put(11,44){\circle*{1}}
\put(7,40){\circle{1}}
\put(7,40){\line(1,1){8}}

\put(31,42){\circle{1}}
\put(29,40){\circle*{1}}
\put(27,38){\circle{1}}
\put(27,38){\line(1,1){4}}

\put(9,30){\circle{1}}
\put(5,26){\circle*{1}}
\put(1,22){\circle{1}}
\put(1,22){\line(1,1){8}}

\put(21,30){\circle{1}}
\put(19,28){\circle*{1}}
\put(17,26){\circle{1}}
\put(17,26){\line(1,1){4}}

\put(19,16){\circle{1}}
\put(17,14){\circle*{1}}
\put(15,12){\circle{1}}
\put(15,12){\line(1,1){4}}

\put(7,45){$\bm{r}_{1}$}

\put(26,42){$\bm{r}_{2}$}

\put(2,28){$\bm{r}_{3}$}

\put(20,26){$\bm{r}_{4}$}

\put(13,15){$\bm{r}_{5}$}

\end{picture}
\hspace{15\unitlength}
\begin{picture}(50,55)

\put(0,20){\vector(1,0){50}}
\put(25,0){\vector(0,1){55}}


\put(1,12){\line(1,0){30}}
\put(1,12){\line(0,1){36}}

\put(31,48){\line(-1,0){30}}
\put(31,48){\line(0,-1){36}}

\put(0,8){$\bm{q}$}
\put(33,48){$\bm{p}$}


\put(37,36){\thicklines\line(-1,0){4}}
\put(37,36){\thicklines\line(0,-1){20}}

\multiput(37,36)(-1,0){5}{\line(0,-1){1}}
\multiput(37,36)(0,-1){20}{\line(-1,0){1}}

\put(29,8){\thicklines\line(0,1){24}}

\multiput(29,9)(0,1){24}{\line(1,0){1}}

\put(28.8,20){\thicklines\line(0,1){12}}
\put(29.2,20){\thicklines\line(0,1){12}}


\put(29,32){\thicklines\line(1,1){4}}
\multiput(30,33)(1,1){3}{\line(1,0){1}}

\put(29,8){\thicklines\line(1,1){8}}
\multiput(30,9)(1,1){8}{\line(-1,0){1}}


\put(11,44){\circle*{1}}

\put(29,40){\circle*{1}}

\put(5,26){\circle*{1}}

\put(19,28){\circle*{1}}

\put(17,14){\circle*{1}}

\end{picture}
\end{center}
\caption{Minimum enclosing rectangle (left) and solution of location problem (right).}
\label{F-MERSLP}
\end{figure}
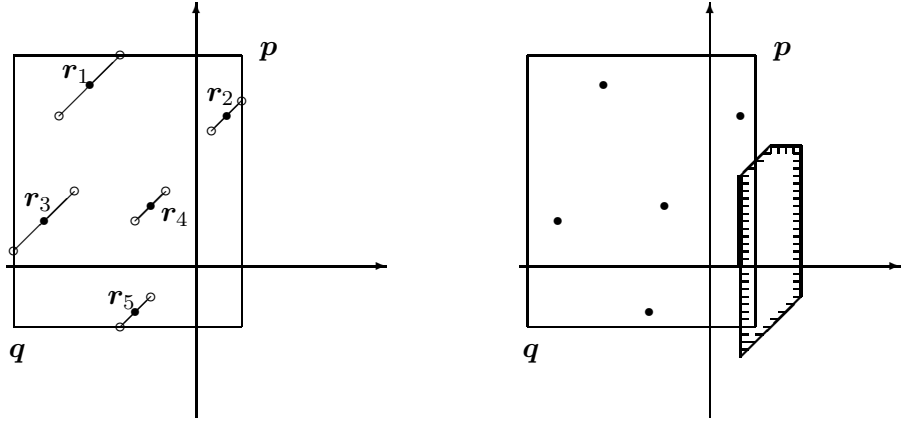

Note that the reduced problem coincides with that examined as an example in the previous section, and thus admits the same solution. We show the solution as a thick vertical segment and the given points as black dots in Fig.~\ref{F-MERSLP} (right).

To conclude this section, we write the solution given by Theorem~\ref{T-xpqxBxxgxh} to problem \eqref{P-Chebyshev} in the usual form.

We first represent the entries of the matrix $\bm{B}^{\ast}=(b_{ij}^{\ast})$ in terms of ordinary operations. It follows from the definition of the asterate operator that
$$
b_{ij}^{\ast}
=
\begin{cases}
\beta_{ij},
&
\text{if $i\ne j$};
\\
\max(\beta_{ij},0),
&
\text{if $i=j$};
\end{cases}
$$
where the numbers $\beta_{ij}$ are calculated as
$$
\beta_{ij}
=
\max_{1\leq k\leq n-1}\max_{\substack{1\leq i_{1},\ldots,i_{k-1}\leq n\\i_{0}=i,i_{k}=j}}(b_{i_{0}i_{1}}+\cdots+b_{i_{k-1}i_{k}}).
$$

Furthermore, we replace the operations of tropical mathematics by arithmetic operations in the rest of the statement of Theorem~\ref{T-xpqxBxxgxh}. By adding definitions for the vectors $\bm{p}$ and $\bm{q}$, we obtain the following statement.
\begin{theorem}
Let $\bm{B}$ be a matrix, and $\bm{g}$ and $\bm{h}$ be vectors such that
\begin{gather*}
\max_{1\leq i,k\leq n}\max_{\substack{1\leq i_{1},\ldots,i_{k-1}\leq n\\i_{0}=i_{k}=i}}(b_{i_{0}i_{1}}+\cdots+b_{i_{k-1}i_{k}})
\leq
0,
\\
\max_{1\leq i,j\leq n}(b_{ij}^{\ast}-h_{i}+g_{j})
\leq
0.
\end{gather*}

Define vectors $\bm{p}=(p_{i})$ and $\bm{q}=(q_{i})$ with elements
$$
p_{i}
=
\max_{1\leq j\leq m}(r_{ij}+w_{j}),
\qquad
q_{i}
=
\min_{1\leq j\leq m}(r_{ij}-w_{j}),
\qquad
i=1,\ldots,n;
$$
and a scalar
$$
\theta
=
\max_{1\leq i,j\leq n}\left((b_{ij}^{\ast}-q_{i}+p_{j})/2,
b_{ij}^{\ast}-h_{i}+p_{j},
b_{ij}^{\ast}-q_{i}+g_{j}\right).
$$

Then, the minimum in \eqref{P-Chebyshev} is $\theta$ and all solutions $\bm{x}=(x_{i})$ are given by
$$
x_{i}
=
\max_{1\leq j\leq n}(b_{ij}^{\ast}+u_{j}),
\qquad
i=1,\ldots,n;
$$
where the numbers $u_{j}$ for each $j=1,\ldots,n$ satisfy the condition
$$
\max(g_{j},p_{j}-\theta)
\leq
u_{j}
\leq
\min\left(-\max_{1\leq i\leq n}(b_{ij}^{\ast}-h_{i}),
\theta-\max_{1\leq i\leq n}(b_{ij}^{\ast}-q_{i})\right).
$$
\end{theorem}


\section*{Conclusions}

The paper was concerned with a new multidimensional tropical optimization problem with a nonlinear objective function and inequality constraints. A complete solution was obtained based on the technique, which reduces the problem to the solution of a linear inequality with a parametrized matrix. The solution is given in a closed form in terms of simple vector operations, which offers low computational complexity and provides for efficient software implementation.

Possible directions of future research include the further extension of the problem to account for new types of objective functions and constraints. The development of new real-world applications of the results is also of interest. 

\section*{Acknowledgments}

The author is very grateful to the three reviewers for their careful reading of a previous draft of this paper. He thanks the reviewers for their valuable comments and illuminating suggestions that have been incorporated in the final version.

\bibliographystyle{utphys}
\bibliography{Complete_solution_of_a_constrained_tropical_optimization_problem_with_application_to_location_analysis}

\end{document}